\documentclass[11pt]{article}
\usepackage{mathrsfs}
\usepackage{amsmath,amsthm}
\usepackage{amsfonts,amssymb,color}
\usepackage{dsfont}
\usepackage{curves}
\usepackage{mathrsfs}
\usepackage{pifont}
\usepackage{graphicx}
\usepackage{float}
\usepackage{hyperref}
\usepackage{cite}

\newtheorem{theorem}{Theorem}[section]
\newtheorem{lemma}[theorem]{Lemma}
\newtheorem{corollary}[theorem]{Corollary}

\theoremstyle{definition}

\newtheorem{remark}[theorem]{Remark}

\usepackage{amsmath}

\date{}

\def\bar{\overline}

\begin{document}
\title{\textbf{Signless Laplacian characterization of cones over disjoint unions of cycles,  edges and isolated vertices}\footnote{ Emails: yejiachang12@163.com (J. Ye), jgqian@xmu.edu.cn (J. Qian),  zstanic@matf.bg.ac.rs (Z. Stani\'c, the corresponding author).}}
\author{\small Jiachang Ye$^{1}$, Jianguo Qian$^{1, 2}$,
 Zoran Stani\' c$^{3}$
  \\\small  $^1$ School of Mathematical Sciences, Xiamen University,  Xiamen, 361005, China
\\\small $^{2}$ School of Mathematics and Statistics, Qinghai Minzu University,  Xining, 810007, China
\\\small $^{3}$ Faculty of Mathematics, University of Belgrade,
Studentski trg 16, 11 000 Belgrade, Serbia} \maketitle

\begin{abstract}
Two graphs are said to be $Q$-cospectral if they share the same signless Laplacian spectrum.
		A simple graph is said to be determined by its signless Laplacian spectrum (abbreviated as DQS)
		if there exists no other non-isomorphic simple graph with the same signless Laplacian spectrum.
		In this paper, we establish the following results:
		 \begin{itemize}
		\item[(1)] Let
		$
		G \cong K_{1} \vee \bigl(C_{k} \cup qK_{2} \cup sK_{1}\bigr),
		$
		with $q,s \geq 1$, $k \geq 4$, and at least $21$ vertices.
		If $k$ is odd, then $G$ is DQS. Moreover, if $k$ is even and $F$ is $Q$-cospectral with $G$, then
		$$
		F \cong G
		\quad \text{or} \quad
		F \cong K_{1} \vee \bigl(C_{4} \cup P_{k-3} \cup P_{3} \cup (q-2)K_{2} \cup sK_{1}\bigr).
		$$
		\item[(2)] Let $G\cong K_1\vee (C_{k_1}\cup C_{k_2}\cup\cdots \cup C_{k_t}\cup qK_2\cup sK_1)$ with $t\ge 2$, $q,s\ge 1$,  $k_i\ge 4$ and at least $33$ vertices. If each $k_i$ is odd, then $G$ is DQS.
		\item[(3)] The graph
		$
		K_{1} \vee \bigl(C_{3} \cup C_{k_{1}} \cup C_{k_{2}} \cup \cdots \cup C_{k_{t-1}} \cup qK_{2} \cup sK_{1}\bigr),
		$
		with $t,q,s \geq 1$ and $k_{i} \geq 3$, is not DQS. Moreover, it is $Q$-cospectral with
		$
		K_{1} \vee \bigl(K_{1,3} \cup C_{k_{1}} \cup C_{k_{2}} \cup \cdots \cup C_{k_{t-1}} \cup qK_{2} \cup (s-1)K_{1}\bigr).
		$
	\end{itemize}
Here $P_{n}$, $C_{n}$, $K_{n}$ and $K_{n-r,r}$ denote the path, the cycle, the complete graph and the complete bipartite graph on $n$ vertices, while $\cup$ and $\vee$ represent the disjoint union and the join of two graphs, respectively. Furthermore, the signless Laplacian spectrum of the graphs under consideration is computed explicitly.

\begin{flushleft}
\end{flushleft}
\end{abstract}

\section{Introduction}
Let $G = (V,E)$ be a finite simple undirected graph.
The number of vertices of $G$ is called the \emph{order} of~$G$, denoted by $n(G)$ (or simply $n$).
The number of edges of $G$ is called the \emph{size} of $G$, denoted by $m(G)$ (or $m$). The \emph{join} of graphs $G$ and $F$, denoted by $G \vee F$, is the graph obtained from the disjoint union of $G$ and $F$ by adding an edge between every vertex of $G$ and every vertex of $F$. If $G$ is the single-vertex graph $K_{1}$, the resulting join is called the \emph{cone} over $F$.

Let $D(G)$ denote the diagonal matrix of vertex degrees of a graph $G$, and let $A(G)$ be its adjacency matrix.
The \emph{signless Laplacian matrix} of $G$ is defined by
$
Q(G) = D(G) + A(G)
$. The eigenvalues of $Q(G)$ are called the \emph{signless Laplacian eigenvalues}, and together they form the \emph{signless Laplacian spectrum} (or simply, the \emph{$Q$-spectrum}) of $G$. When the context is clear, the prefix $Q$- will be omitted.

Two graphs are said to be \emph{$Q$-cospectral} if they have the same $Q$-spectrum.
A graph $G$ is said to be \emph{determined by its signless Laplacian spectrum} (abbreviated as \emph{DQS}) if every graph $F$ that is $Q$-cospectral with $G$ is necessarily isomorphic to $G$.

The question of whether a graph is determined by the spectrum of a prescribed matrix is one of the most intriguing and extensively studied problems in spectral graph theory. Based on a combination of theoretical observations, exhaustive searches of graphs of relatively small order, and comparisons of the spectra of the adjacency, Laplacian and signless Laplacian matrices, the authors of \cite{towI,towII} concluded that cospectrality occurs least frequently with respect to the signless Laplacian spectrum. According to these references, this observation suggests that spectral graph theory based on the signless Laplacian may be more effective than the other two approaches, which have historically received much greater attention in the literature.

We believe our readers will agree that determining whether a fixed graph is uniquely identified by its spectrum is a challenging problem, even in the case of graphs with relatively simple structures. For some basic results and further developments, we refer to~\cite{3,4}, while comprehensive treatments of DQS graphs are available in~\cite{towI,towII}. Moreover, specific classes of DQS graphs have been established in~\cite{CSS,Liu-ELA,Liu2} and the references therein. For a broader perspective, including a survey and thorough discussion on spectral determination problems, see~\cite{18,towI,towII}.

In this paper, we build upon the research initiated in \cite{Ye2025DAM} and further developed in \cite{Ye2025Signless} by studying in greater depth the $Q$-cospectrality of cones $K_1 \vee F$, where $F$ is taken to be a disjoint union of cycles, edges and isolated vertices. This class of graphs provides a  testing ground for understanding how structural properties of the base graph $F$ influence the spectral behaviour of the cone. An intriguing phenomenon emerges: Even relatively minor modifications in the structure of $F$, such as the absence of a particular component type or the exclusion of cycles of specific lengths, can lead to substantially different outcomes. Depending on these changes, the resulting graph $K_1 \vee F$ may either be uniquely determined by its $Q$-spectrum (DQS) or, in contrast, create families of non-isomorphic graphs that are nevertheless $Q$-cospectral. These results illustrate the delicate balance between combinatorial structure and spectral invariants, and highlight how the interplay between components of $F$ governs the spectral uniqueness of the cone.

The cases in which $F$ does not contain edges or isolated vertices have already been resolved in \cite{Ye2025DAM} and~\cite{Ye2025Signless}, respectively. It is worth mentioning that preliminary results related to the former case had appeared earlier in  \cite{Liu-wheel,Ye2024LAA,Liu21}. In \cite{Liu-wheel,Liu21}, attention was restricted to the case where $F$ contains only cycle components, thereby establishing a first step toward the general problem. On the other hand, in \cite{Ye2024LAA}, the authors focused on the situation where $F$ consists of a single cycle together with an arbitrary number of isolated vertices, thus providing complementary insights.

Concerning the background, we note that a related line of research on spectral determination was carried out in~\cite{CSS}, where the authors investigated disjoint unions of paths and cycles. In addition to the already mentioned references, the present paper also connects to a number of works on the spectral determination of graph products, including \cite{MH3,Liu11,Liu1,sage,Wang11,Wheel-DLS,Zhou11}.  In particular, Theorem~\ref{12t} below extends this body of work by generalizing previous results on the $Q$-spectral determination of friendship graphs and star graphs established in~\cite{Liu-ELA}.

To formulate our results, we introduce some basic notation. We denote by $P_n$, $K_n$, $K_{n-r,r}$, and $C_n$ the path, the complete graph, the complete bipartite graph, and the cycle of order $n$, respectively. The \emph{disjoint union} of two graphs $G$ and $F$ is denoted by $G \cup F$, and the disjoint union of $r$ copies of a graph $G$ is written as $rG$. Our main contributions read as follows.

 \begin{theorem}\label{11t} Every graph $G\cong K_1\vee (C_k\cup qK_2\cup sK_1)$, with $q,s\ge 1$, $k\ge 4$ and at least $21$ vertices, is DQS if $k$ is odd. Moreover, if $k$ is even and  $F$ is $Q$-cospectral with $G$, then $$F\cong G \quad \text{or} \quad F\cong K_1\vee (C_{4}\cup P_{k-3}\cup P_{3}\cup (q-2)K_2 \cup sK_1).$$
\end{theorem}

 \begin{theorem}\label{12t} The graph $K_1\vee (C_{k_1}\cup C_{k_2}\cup\cdots \cup C_{k_t}\cup qK_2\cup sK_1)$, with $t\ge 2$, $q,s\ge 1$,  $k_i\ge 4$ $(1\le i \le t)$ and at least $33$ vertices, is DQS whenever every $k_i$ is odd.
\end{theorem}

\begin{theorem}\label{13t} Every graph $K_1\vee (C_3\cup C_{k_1}\cup C_{k_2}\cup\cdots \cup C_{k_{t-1}}\cup qK_2\cup sK_1)$, with $t,q,s\ge 1$, $k_i \ge 3$, is not DQS. Moreover, it is $Q$-cospectral with $K_1\vee (K_{1,3}\cup C_{k_1}\cup C_{k_2}\cup\cdots \cup C_{k_{t-1}}\cup qK_2\cup (s-1)K_1)$.
\end{theorem}

Note that the second theorem does not generalize the first one, since they impose different requirements on the order of the graph. We point out that the $Q$-spectra of the graphs under consideration are explicitly computed, and several auxiliary statements are formulated in a form that also accommodates the case where some cycles degenerate into digons, i.e., two parallel edges connecting the same pair of vertices. In this way, the corresponding results naturally extend to the setting of multigraphs.

The remainder of the paper is organized as follows. Section~\ref{sec2} introduces additional terminology and notation, as well as some known results. In Section~\ref{sec3}, we compute the $Q$-spectrum of cones over some particular graphs and establish auxiliary results concerning the largest $Q$-eigenvalue, which serve as a basis for the proof of Theorem~\ref{13t}. In particular, this section establishes the $Q$-spectrum of the main structure studied in this paper: A cone over a disjoint union of cycles, edges and isolated vertices. Section~\ref{sec4} focuses on the spectral moments of the signless Laplacian matrix for cones considered in the previous section, as well as for certain related conical graphs. Building on these findings, the proofs of Theorems~\ref{11t} and~\ref{12t} are carried out in Sections~\ref{sec5} and~\ref{sec6}, respectively.

\section{Preliminaries}\label{sec2}

Let $N_{G}(v)$   and $d_{G}(v)$  be the set of neighbours of a vertex $v$
 and the degree of  $v$ in a graph $G$, respectively.  $G[Y]$  denotes the subgraph induced by $Y$, where $Y\subset V(G)$. The graphs obtained by deleting edge $e$  and vertex $v$ of $G$ are denoted by $G-e$ and   $G-v$, respectively.  The eigenvalues of $A(G)$ and $Q(G)$ are denoted by $$\lambda_1(G)\ge \lambda_2(G)\ge \cdots\ge \lambda_n(G)  \text{\quad and \quad} \chi_{1}(G)\ge \chi_{2}(G)\ge \cdots\ge \chi_{n}(G),$$ respectively. Since $Q(G)$ is positive semidefinite,  $\chi_n(G)\ge 0$ is always true (see \cite{Book-Stanic}). The $Q$-spectrum of $G$ is denoted by $S_Q(G)$; of course, it is considered as a multiset.

Henceforth,  $m_{G}(I^*)$ denotes the number of the eigenvalues in the interval $I^*$ in~$S_Q(G)$. In particular, $m_{G}(\rho)$ denotes the multiplicity of the eigenvalue~$\rho$ in~$S_Q(G)$. Besides, $\rho^{(r)}$ denotes either $r$ copies of the real number $\rho$, or a vector of length $r$ with all entries equal to $\rho$, depending on the context.

 We assume that  the degree $d_i(G)$  is attained by a vertex $v_i$ $(1\leq i\leq n)$ of $G$. In this context, $n_r(G)$  denotes the number of vertices of degree $r$ (for short, \textit{$r$-vertices}) in  $V(G)\setminus \{v_1\}$, that is, $$n_r(G)=|\{v\,:\,  d_G(v)=r ~\text{and}~v\in V(G)\setminus \{v_1\}\}|,$$
where $v_1$ is a vertex with the maximum degree. Throughout this paper, we assume that  vertex degrees are arranged in non-decreasing order:
$d_n \le d_{n-1} \le \cdots \le d_1$.

For notational simplicity, we occasionally omit the graph $G$ in the preceding expressions.

Let $\rho_{i}(N)$, $1\le i\le n$, denote  the $i$th largest eigenvalue of an $n\times n$  real symmetric matrix~$N$. Suppose now that the columns of $N$ are indexed by
$Y=\{1,2,\ldots,n\}$. For a partition ${Y_1,Y_2,\ldots,Y_r}$ of $Y$, we set
\begin{equation*}
 N=   \begin{bmatrix}
    N_{1,1} & \ldots & N_{1,r}\\
    \vdots & \ddots & \vdots\\
    N_{r,1} & \ldots & N_{r,r}
    \end{bmatrix}
\end{equation*}
where $N_{i,j}$ denotes the block of $N$ formed by the rows in $Y_i$ and the columns in $Y_j$. If $c_{i,j}$ denotes the average row sum in $N_{i,j}$, then the matrix $M=[c_{i,j}]$ is known as the \textit{quotient matrix} of $N$. If, for every $i, j$,  all row sums of $N_{i,j}$ are equal, then the corresponding partition is called \textit{equitable}.

\begin{lemma} [\cite{BroSpe}]\label{equitable}
Let $N$ be a non-negative irreducible real symmetric matrix, and $M$ the quotient matrix of an equitable partition of  $N$. If $\rho$ is an eigenvalue of $M$, then $\rho$ is also an eigenvalue of  $N$. Furthermore, the largest eigenvalues of $N$ and $M$ coincide.
 \end{lemma}

The remainder of this section presents some results regarding $Q$-eigenvalues, which can be found in the relevant literature.

\begin{lemma}[\cite{Heu1}] \label{DeleteEdge-lem} Let $G$ be a graph of order $n$ $(n\geq 3)$ and $e\in E(G)$. Then
$\chi_i(G)\ge \chi_i(G-e)$ holds for $1\le i \le n$.
\end{lemma}

\begin{lemma}[\cite{Liu-wheel}] \label{DeleteVertex-lem} Let $G$ be a graph of order $n$ and  $v$  a vertex of degree $n-1$ in $G$. Then
$$\chi_{i}(G)-1\ge \chi_{i}(G-v)\ge \chi_{i+1}(G)-1$$ holds for $1\le i \le n-1$.
\end{lemma}


\begin{lemma}[\cite{18}]\label{mul-0-lem} The multiplicity of the eigenvalue zero in the $Q$-spectrum of a graph is equal to the number of bipartite components.
\end{lemma}

\begin{lemma}[\cite{Liu21}]\label{d1-d2-lem}  If $G$ is a graph of order $n$ $(n\geq 12)$ with $\chi_{1}>n>5\ge \chi_{2}\ge \chi_{n}>0$, then $G$ is connected with $d_{1}\geq n-3$  and $d_{2}\leq4$.
\end{lemma}

\begin{lemma}[\cite{Ye2024LAA}]\label{d1-lem}  Let $G$ be a  graph of order $n$ $(n\ge 16)$ with $\chi_{1}(G)>n>5> \chi_{2}(G)\ge \chi_{n}(G)>0$.   If  $F$ and $G$ are $Q$-cospectral, then $F$ is  connected, along  with  $d_{2}(F)\leq 4$ and  $d_1(F)=d_1(G)\in \{n-1,n-2\}$.
\end{lemma}

\begin{lemma}[\cite{Ye2024LAA}]\label{n4-lem}  Let $G$  be a graph  of order $n$ such that $\chi_2<5$, $d_1=n-1$ and  $d_2\le 4$. If $n\ge 21$,  then $n_4\le 1$ and  every $4$-vertex is adjacent to at most one $3$-vertex.
\end{lemma}

\begin{lemma}[\rm\cite{Ye2025DAM}]\label{leq5-lem}  Let $G$  be a graph  of order $n$ such that $\chi_2\le 5$,   $d_1=n-1$ and  $d_2\le 4$.
	\begin{itemize}
       \item[(i)] If $n\ge 14$, then every two $4$-vertices (if any) are non-adjacent;
		\item[(ii)] If $n\ge 22$,  then every $4$-vertex is adjacent to at most one $3$-vertex;
		\item[(iii)] If $n\ge 22$, then any pair of $4$-vertices $v$ and $w$ (if any) satisfies $N_G(v)\cap N_G(w)=\{v_1\}$.
	\end{itemize}
\end{lemma}

Let $Z_n$ denote a tree with $n\ge 4$ vertices obtained by duplicating an endvertex of the path $P_{n-1}$.

\begin{lemma}[\cite{Ye2024LAA,Ye2025DAM}]\label{leq5-lem-2}  For a graph $G$, if $\chi_2\le 5$, then $G$ does not contain $K_1\vee (Z_6\cup 15K_1)$ as a subgraph.  Moreover, if $\chi_2<5$, then $G$ does not contain any of $K_1\vee (Z_6\cup   14K_1)$,  or $K_1\vee (K_{1,3}\cup C_r)$ $(r\ge 3)$, or $K_1\vee (C_{r_1}\cup C_{r_2})$ $(r_1,r_2\ge 3)$ as a subgraph.
\end{lemma}

\begin{lemma}[\cite{Ye2024LAA}]\label{23l}  Let $G$ be a connected  graph of order $n$ with $d_2\leq 4$. If either $d_1\ge 11>1=d_n$   or $d_1\ge 8>d_n\ge 2$, then $\chi_1\le d_1+3$.
\end{lemma}

\section{$Q$-eigenvalues of certain cones and  proof of Theorem \ref{13t}}\label{sec3}

This section is concerned with the $Q$-eigenvalues of  $$K_1\vee (C_{k_1}\cup C_{k_2}\cup\cdots \cup C_{k_t}\cup qK_2\cup sK_1)\quad \text{and} \quad K_1\vee (K_{1,3}\cup C_{k_1}\cup C_{k_2}\cup\cdots \cup C_{k_{t-1}}\cup qK_2\cup (s-1)K_1).$$
In the former graph, we permit the possibility that certain cycles degenerate into digons, which are addressed in a separate statement.


We begin with two auxiliary results. The first concerns the multiplicity of 1 in the $Q$-spectrum of a graph with pendant vertices sharing a common neighbourhood.

\begin{lemma}[\cite{Ye2025DAM}]\label{eigenvalue1-lem} Let $G$ be a  graph of order $n~(n\ge 2)$ and $V(G)=\{w_i~:~ 1\le i \le n\}$. If $N_G(w_1)=N_G(w_2)=\cdots =N_G(w_r)=\{w_{r+1}\}$ where $2\le r+1\le n$, then $m_{G}(1)\ge r-1$ and ${\boldsymbol x_j}=(x_{j}(w_1),x_{j}(w_2),\ldots,x_{j}(w_n))^{\intercal}$, $1\leq j\leq r-1$, are $r-1$ linearly independent eigenvectors of $Q(G)$ corresponding to eigenvalue $1$, where $x_{j}(w_j)=1$, $x_{j}(w_{j+1})=-1$ and $x_{j}(w_i)=0$ for $i\notin \{j,j+1\}$.
\end{lemma}

The \textit{coalescence graph}, denoted by $G(u)\odot F(v)$,  is obtained from $G$ and $F$  by merging $u$ ($u\in V(G)$) and $v$ ($v\in V(F)$). The next lemma tells us about the multiplicities of $1$ and $3$ in the $Q$-spectrum of a special coalescence.

\begin{lemma}\label{eigenvalue1-3-lem} Let $G$ be a graph with vertex $v$, $H\cong K_{1}\vee rK_2$ with vertex $u$ of degree $2r$ $(r\ge 1)$ and $F\cong G(v)\odot H(u)$. Then $m_F(1)\ge r$ and $m_F(3)\ge r-1$.
\end{lemma}
\begin{proof}
Suppose $V(F)=\{w_i~:~ 1\le i\le n\}$, where $d_F(w_i)=2$ for $1\le i \le 2r$ and $w_{2j}w_{2j-1}\in E(F)$ for $1\leq j\leq r$.
 We first deal with the eigenvalue 1. Let $${\boldsymbol \alpha_1}=\big(1,-1,0^{(n-2)}\big)^\intercal,  {\boldsymbol \alpha_2}=\big(0,0,1,-1,0^{(n-4)}\big)^\intercal, \ldots, {\boldsymbol \alpha_r}=\big(0^{(2r-2)},1,-1,0^{(n-2r)}\big)^\intercal.$$
We can easily check that this comprises a set of linearly independent eigenvectors associated with the eigenvalue $1$ of $Q(F)$. Thus, $m_{F}(1)\ge r$.

Next,  $${\boldsymbol \beta_1}=(1,1,-1,-1,0^{(n-4)})^\intercal,  {\boldsymbol \beta_2}=(0^{(2)},1,1,-1,-1,0^{(n-6)})^\intercal, \ldots, {\boldsymbol \beta_{r-1}}=(0^{(2r-4)},1,1,-1,-1,0^{(n-2r)})^{\intercal}.$$
are linearly independent eigenvectors of the same matrix corresponding to the eigenvalue $3$,  when $r\ge 2$.
\end{proof}

We now compute the $Q$-spectrum of a cone over cycles, edges and isolated vertices.

\begin{lemma}\label{SQG-lem}Let $G\cong K_1\vee (C_{k_1}\cup C_{k_2}\cup\cdots \cup C_{k_t}\cup qK_2\cup sK_1)$, $t,q,s\ge 1$ and $k_i\ge 3$ $(1\le i\le t)$. If $n$ is the order of $G$, then
	$$S_Q(G)=\left\{\rho_1,\rho_2,\rho_3,\rho_4, 5^{(t-1)}, 3^{(q-1)}, 1^{(s+q-1)}, 3+2\cos\frac{2j\pi}{k_i}~:~ 1\le j\le k_i-1, 1\le i\le t\right\},$$
	where $\rho_i$, $1\le i \le 4$, are the  roots of $\rho^4-(n+8)\rho^3+(8n+15)\rho^2+(4q+4s-19n+4)\rho+12n-4q-12s-12$. In addition, these particular roots satisfy $\rho_1>n>5>\rho_2>4>3>\rho_3>2>1>\rho_4>0$.
\end{lemma}

\begin{proof} We suppose that $V(G)=\{w_i\,:\, 1\leq i\leq n\}$, where $d_G(w_n)=n-1$, $d_G(w_i)=1$ for $1\leq i\leq s$, $d_G(w_i)=2$ for $s+1\leq i\leq s+2q$, and $w_{s+2j}w_{s+2j-1}\in E(G)$ for $1\leq j\leq q$. Combining with Lemmas \ref{eigenvalue1-lem} and \ref{eigenvalue1-3-lem}, we find $m_G(1)\ge s+q-1$ and $m_G(3)\ge q-1$. In what follows, we construct the eigenvectors for the remaining eigenvalues.

We first deal with the eigenvalue 5. Let ${\boldsymbol \gamma_j}=(\gamma_{j}(w_1),\gamma_{j}(w_2),\ldots,\gamma_{j}(w_n))^\intercal$ for $1\leq j\leq t-1$, where $$\gamma_{j}(w_i)=\left\{\begin{array}{rl}-k_{j+1},&\text{for}~w_i\in V(C_{k_1}),\\ k_1,& \text{for}~w_i\in V(C_{k_{j+1}}),\\ 0,&\text{otherwise}.\end{array}\right.$$

It is easy to check that this comprises a set of linearly independent eigenvectors associated with the eigenvalue $5$ of $Q(G)$. Thus, $m_{G}(5)\ge t-1$.

At this point we  denote $V(C_{k_i})=\{a_{i1},a_{i2},\ldots,a_{ik_i}\}$ ($1\leq i\leq t$). Suppose $A(C_{k_{i}}){\boldsymbol \eta}=\lambda {\boldsymbol \eta}$ (${\boldsymbol \eta}\ne {\bf 0}$), i.e., ${\boldsymbol \eta}$ is an eigenvector for $A(C_{k_{i}})$ corresponding  to the eigenvalue $\lambda$. It is well known that $\lambda=2\cos(2j\pi/k_i)$, $0\leq j\leq k_i-1$. Besides, it is easy to prove that if $\lambda\neq 2$  (i.e. $j\neq 0$), then we can choose ${\boldsymbol \eta}=(\eta(a_{i1}),\eta(a_{i2}),\ldots,\eta(a_{ik_{i}}))^{\intercal}$ as  an  eigenvector corresponding  to $\lambda$  such that $\sum_{l=1}^{k_{i}}\eta(a_{il})=0$. Thus, by excluding $\lambda=2$, we  construct  $\bar{{\boldsymbol \eta}}=(\bar{\eta}(w_{1}), \bar{\eta}(w_{2}),\ldots,\bar{\eta}(w_{n}))^\intercal$ with
$$\bar{\eta}(w)=\left\{\begin{array}{rl}\eta(w),&\text{for}~w\in V(C_{k_{i}}),\\ 0,& \text{otherwise},\end{array}\right.$$
as an eigenvector for the eigenvalue $\lambda+3$. To see this, one may observe that $Q(C_{k_{i}})+I_{k_{i}}=A(C_{k_{i}})+3I_{k_{i}}$ is a principal submatrix of $Q(G)$. Therefore, $3+2\cos(2j\pi/k_i)$, $1\leq j\leq k_{i}-1$,  are the eigenvalues of $Q(G)$, for every $i~(1\leq i\leq t)$.

It remains to deal with the $Q$-eigenvalues denoted by $\rho_i$, $1\le i \le 4$, in the  statement. The polynomial $\rho^4-(n+8)\rho^3+(8n+15)\rho^2+(4q+4s-19n+4)\rho+12n-4q-12s-12$ can be derived from the characteristic polynomial of an equitable quotient matrix of $Q(G)$.
 It is directly verified  that the corresponding roots satisfy $\rho_1>n>5>\rho_2>4>3>\rho_3>2>1>\rho_4>0$. Moreover, $${\boldsymbol \psi_i}=\Big((\rho_i-3)(\rho_i-5)^{(s)},(\rho_i-1)(\rho_i-5)^{(2q)},(\rho_i-1)(\rho_i-3)^{(n-s-2q-1)},(\rho_i-1)(\rho_i-3)(\rho_i-5)\Big)^\intercal$$
	is an eigenvector of $Q(G)$ corresponding to the eigenvalue $\rho_i$ for $1\le i \le 4$.
	
Finally, another scenario must be considered: In certain cases, some of the computed eigenvalues may coincide, raising the question of whether their corresponding eigenvectors remain linearly independent.
From above proof we know that $\rho_1, \rho_2$, $\rho_3$ and $\rho_4$ are mutually distinct and do not belong to $\{5, 3, 1\}$. What's more,
$1 \leq 3 + 2\cos\frac{2j\pi}{k_i} < 5,$
for $1 \leq j \leq k_i - 1$ and $1 \leq i \leq t$.
In this context, if the equality
$$3 + 2\cos\frac{2j\pi}{k_i} = \rho_r$$
holds for some $r$, $j$ and $i$, then the corresponding eigenvectors are linearly independent by way of their construction. The same conclusion holds if the left-hand side is equal to either $3$ or $1$. The proof is completed.
\end{proof}

\begin{remark}\label{r3.2} Note that  $3+2\cos(2j\pi/k_i)=1$ holds if and only if $k_i$ is even and $j=k_i/2$. Thus, $m_G(1)=s+q-1+c_{e}(C)$, where $c_{e}(C)$ denotes the number of even cycles among $C_{k_i}$.  Besides, we also have $$n<\chi_1(G)=\rho_1<n+2,~ 0<\chi_n(G)=\rho_4<1,~\chi_2(G)<5~ (\text{when}~ t=1),~\text{and}~\chi_2(G)=5~(\text{when}~t\ge 2).$$ Last but not least, for any given order $n$, $\chi_1(G)$ depends on the numbers $q$ and $s$, but does not depend on the number of cycles nor the length of each cycle.
\end{remark}

We now extend the context to include the possibility of cycles $C_2$, which are regarded as digons consisting of two parallel edges connecting the same pair of vertices. In this setting, the degree of a vertex is defined as the number of edges incident with it, and the entries of the adjacency matrix are equal to the number of the corresponding edges.

\begin{lemma}\label{GG*-lem}Let $G^*\cong K_1\vee (C_{k_1}\cup C_{k_2}\cup\cdots \cup C_{k_t}\cup qK_2\cup sK_1)$, with $t,q,s\ge 1$ and $k_i\ge 2$ $(1\le i\le t)$. If $n$ is the order of $G^*$, then $\chi_1(G^*)$
is equal to the largest root of $$\rho^4-(n+8)\rho^3+(8n+15)\rho^2+(4q+4s-19n+4)\rho+12n-4q-12s-12.$$
\end{lemma}
\begin{proof} It is easy to see that $Q(G^*)$ has an equitable quotient matrix
\begin{equation*}
 N=   \begin{bmatrix}
    n-1 & n-1-2q-s & 2q & s\\
    1 & 5 & 0 & 0\\
    1 & 0 & 3 & 0\\
    1 & 0 & 0 & 1
    \end{bmatrix}.
\end{equation*}
We find that the characteristic polynomial of $N$ is $$\rho^4-(n+8)\rho^3+(8n+15)\rho^2+(4q+4s-19n+4)\rho+12n-4q-12s-12.$$
The  desired result follows from Lemma \ref{equitable}.
\end{proof}

Combining with Lemma \ref{SQG-lem}, Remark \ref{r3.2} and Lemma \ref{GG*-lem}, we  immediately obtain the following corollary.
\begin{corollary}\label{GG*-coro} Consider a simple graph $G\cong K_1\vee (C_{k_1}\cup C_{k_2}\cup\cdots \cup C_{k_t}\cup qK_2\cup sK_1)$ and a multigraph $G^*\cong K_1\vee (C_{k'_1}\cup C_{k'_2}\cup\cdots \cup C_{k'_z}\cup qK_2\cup sK_1)$, where $t,z,q,s\ge 1$, $k_i\ge 3$ $(1\le i\le t)$, and $k'_i\ge 2$ $(1\le i\le z)$. If $n(G)=n(G^*)$, then
$\chi_1(G)=\chi_1(G^*)$.
\end{corollary}

We proceed with computing the $Q$-spectrum of the latter class introduced at the beginning of the section.

\begin{lemma}\label{SQF-lem} Let $F\cong K_1\vee (K_{1,3}\cup C_{k_1}\cup C_{k_2}\cup\cdots \cup C_{k_{t-1}}\cup qK_2\cup (s-1)K_1)$, with $q,s\geq 1, t\geq 2$ and $k_i\geq3$. If $n$ is the order of $F$, then
	$$S_Q(F)=\left\{\rho_1,\rho_2,\rho_3,\rho_4, 1^{(s+q-1)},2^{(2)},3^{(q-1)},5^{(t-1)}, 3+2\cos\frac{2j\pi}{k_i}~:~ 1\le j\le k_i-1, 1\le i\le t-1\right\},$$
	where $\rho_i$, $1\le i\le 4$, are the roots of $$\rho^4-(n+8)\rho^3+(8n+15)\rho^2+(4q+4s-19n+4)\rho+12n-4q-12s-12.$$ These roots satisfy $\rho_1>n>5>\rho_2>4>3>\rho_3>2>1>\rho_4>0$.
	
	In particular, for $F^*\cong K_1\vee (K_{1,3}\cup qK_2\cup  (s-1)K_1)$, $q,s\ge 1$, we have $$S_Q(F^*)=\left\{\rho_1,\rho_2,\rho_3,\rho_4,1^{(s+q-1)},2^{(2)},3^{(q-1)}\right\}.$$
\end{lemma}
\begin{proof} We only consider the case when $s,t\ge3$ and $q\ge 2$, as the remaining possibility is treated by following the same lines. For convenience, we set $V(F)=\{w_i\,:\, 1\leq i\leq n\}$, with $d_F(w_n)=n-1$, $d_F(w_{n-4})=d_F(w_{n-3})=d_F(w_{n-2})=2, d_F(w_{n-1})=4$, $d_F(w_i)=1$ for $1\leq i\leq s-1$, $d_F(w_i)=2$ for $s\leq i\leq s+2q-1$,  and  $w_{s-1+2j}w_{s-2+2j}\in E(F)$ for $1\le j \le q$. Note that $K_{1,3}$ is a component of $F-w_n$ with $V(K_{1,3})=\{w_{n-4},w_{n-3},w_{n-2},w_{n-1}\}$.
	
The eigenvectors ${\boldsymbol \alpha_i}$, $1\leq i\leq s+q-2$, for $1$,  and the eigenvectors ${\boldsymbol \beta_i}$, $1\leq i\leq q-1$, for $3$ are constructed as  Lemmas~\ref{eigenvalue1-lem} and \ref{eigenvalue1-3-lem}, whereas the eigenvectors ${\boldsymbol \gamma_i}$, $1\leq i\leq t-2$, for $5$ and all the eigenvectors for $3+2\cos\frac{2j\pi}{k_i}, 1\le j\le k_i-1, 1\le i\le t-1$, are constructed as in the proof of Lemma~\ref{SQG-lem}. 	

Also,
$${\boldsymbol \alpha_{s+q-1}}
=\big(2,0^{(n-6)},-1,-1,-1,1,0\big)^\intercal,$$
$${\boldsymbol \gamma_{t-1}}
=\big(0^{(s+2q-1)},\big(-{6}/{k_1}\big)^{(k_1)},0^{(k_2)},\ldots,0^{(k_{t-1})},
1,1,1,3,0\big)^\intercal,$$
and
$$
\big(0^{(n-5)},-1,1,0,0,0\big)^\intercal,~ \big(0^{(n-5)},-1,0,1,0,0\big)^\intercal,$$
are associated with $1$, $5$ and $2^{(2)}$, respectively.

 We next  show that $\rho_1$, $\rho_2$, $\rho_3$ and $\rho_4$ are the eigenvalues of $Q(F)$.  Suppose $Q(F){\boldsymbol \psi}=\rho{\boldsymbol \psi}$ (${\boldsymbol \psi}\ne {\bf 0}$), where ${\boldsymbol \psi}=(\psi(w_1),\psi(w_2),\ldots,\psi(w_n))^\intercal$. Then for $w_n$ (the vertex with maximum degree), we have
	\begin{align}\label{e3.1}  (\rho-n+1)\psi(w_n)=\sum_{i=1}^{n-1}\psi(w_i). \end{align}
For the $s-1$ pendant vertices $w_1, w_2,\ldots, w_{s-1}$, we have
\begin{align*}(\rho-1)\psi(w_i)=\psi(w_n), \quad 1\le i \le s-1.
\end{align*}
So, if $\rho \ne 1$, then
\begin{align}\label{e3.2}
		\psi(w_i)=\frac{\psi(w_n)}{\rho-1}, \quad 1\leq i\leq s-1.
	\end{align}
For the vertices $w_{s},w_{s+1},\ldots,w_{s+2q-1}$ of degree 2, we have
\begin{equation*}\left\{
		\begin{array}{ll} (\rho-2)\psi(w_{s+2j-1})=\psi(w_{s+2j-2})+\psi(w_n),\\
			(\rho-2)\psi(w_{s+2j-2})=\psi(w_{s+2j-1})+\psi(w_n),  \quad 1\leq j\leq q.
		\end{array}
		\right.\end{equation*}
Gathering the previous equations, we  obtain $(\rho-1)\psi(w_{s+2j-1})=(\rho-1)\psi(w_{s+2j-2})$.  Hence, if $\rho\ne 1$ then $\psi(w_{s+2j-1})=\psi(w_{s+2j-2})$, $1\leq j\leq q$. Furthermore, if $\rho\ne 3$, then we get that
\begin{align}\label{e3.3}
\psi(w_{i})=\frac{\psi(w_n)}{\rho-3}, \quad s\leq i\leq s+2q-1.
\end{align}
We denote $V(C_{k_i})=\{a_{i1},a_{i2},\ldots,a_{ik_i}\}$, $1\leq i\leq t-1$. For the vertices in $V(C_{k_i})$, we have
	\begin{equation*}\left\{
		\begin{array}{ll} (\rho-3)\psi(a_{i1})=\psi(a_{i,k_i})+\psi(w_n)+\psi(a_{i2}),\\
			(\rho-3)\psi(a_{i2})=\psi(a_{i1})+\psi(w_n)+\psi(a_{i3}),\\
			\quad\quad\quad\quad \vdots \\
			(\rho-3)\psi(a_{i,k_i})=\psi(a_{i,k_i-1})+\psi(w_n)+\psi(a_{i1}).
		\end{array}
		\right.\end{equation*}
	We set $\psi(a_{i,k_i})=\psi(a_{i,k_i-1})=\cdots =\psi(a_{i,1})$. If $\rho\ne 5$, then we have
	\begin{align}\label{e3.4}
		\psi(a_{ij})=\frac{\psi(w_n)}{\rho-5},\quad 1\leq j\leq k_i, \,\, 1\leq i\leq t-1.
	\end{align}
For the vertices in $V(K_{1,3})=\{w_{n-4},w_{n-3},w_{n-2},w_{n-1}\}$, we have
\begin{equation}\label{e3.5}\left\{
\begin{array}{ll}
\rho\psi(w_{n-1})=4\psi(w_{n-1})+\psi(w_{n-2})+\psi(w_{n-3})+\psi(w_{n-4})+\psi(w_n),\\
\rho\psi(w_{n-2})=2\psi(w_{n-2})+\psi(w_{n-1})+\psi(w_n),\\
\rho\psi(w_{n-3})=2\psi(w_{n-3})+\psi(w_{n-1})+\psi(w_n),\\
\rho\psi(w_{n-4})=2\psi(w_{n-4})+\psi(w_{n-1})+\psi(w_n).\\
\end{array}
\right.
\end{equation}
We may take that $\psi(w_n)=1$. From equality \eqref{e3.5}, we deduce that if $\rho\notin\{1,2,5\}$, then
\begin{align}\label{e3.6}
\psi(w_{n-1})=\frac{\rho+1}{(\rho-1)(\rho-5)}
\end{align}
and
\begin{align}\label{e3.7}
\psi(w_{n-2})=\psi(w_{n-3})=\psi(w_{n-4})=\frac{\rho-3}{(\rho-1)(\rho-5)}.
\end{align}

	Combining \eqref{e3.1}, \eqref{e3.2}, \eqref{e3.3}, \eqref{e3.4}, \eqref{e3.6} and \eqref{e3.7} and inserting $\psi(w_n)=1$, we find
	$$\rho-n+1=(s-1)\frac{1}{\rho-1}+ 2q\frac{1}{\rho-3}+\frac{4(\rho-2)}{(\rho-1)(\rho-5)}+(n-s-2q-4)\frac{1}{\rho-5}.$$
	Equivalently, $$\rho^4-(n+8)\rho^3+(8n+15)\rho^2+(4q+4s-19n+4)\rho+12n-4q-12s-12=0.$$ It is verified directly that the corresponding roots satisfy the desired inequalities. Therefore,
 $\rho_i$ appears as a $Q$-eigenvalue of $F$ and an associated eigenvector is	$${\boldsymbol \psi_i}=\Big(\big(\frac{1}{\rho_i-1}\big)^{(s-1)},\big(\frac{1}{\rho_i-3}\big)^{(2q)},\big(\frac{1}{\rho_i-5}\big)^{(n-2q-s-4)},
	\big(\frac{\rho_i-3}{(\rho_i-1)(\rho_i-5)}\big)^{(3)}, \frac{\rho_i+1}{(\rho_i-1)(\rho_i-5)},   1\Big)^\intercal,$$ for
 $1\le i \le 4$.

Linear independence in matching cases are considered as in the proof of Lemma~\ref{SQG-lem}. Finally, the result for $F^*$ is extracted by setting $t=1$.
\end{proof}

Ultimately, we record the proof of Theorem \ref{13t}.

\medskip\noindent{\bf \textit{Proof of Theorem \ref{13t}}.} From Lemmas~\ref{SQG-lem} and~\ref{SQF-lem}, the desired result follows immediately upon setting $k_t = 3$.
\qed

\section{$Q$-spectral moments of certain cones}\label{sec4}

For a graph $G$ and a non-negative integer $r$, the sums
\[
S_r(G) = \sum_{i=1}^{n} \lambda_i^r(G) \quad \text{and} \quad T_r(G) = \sum_{i=1}^{n} \chi_i^r(G)
\]
are called the $r$th \emph{spectral moment} and the $r$th \emph{$Q$-spectral moment} of $G$, respectively.

In this section, we compute the difference between the $Q$-spectral moments of certain cones. To this end, we first introduce some additional notation and recall a well-known lemma concerning $Q$-spectral moments.

We denote by $\varsigma_G(H)$ the number of subgraphs of $G$ that are isomorphic to $H$. Moreover, for a vertex $v \in V(G)$, let $t_G(v)$ be the number of triangles containing $v$. We also define the quantities
\begin{equation}\label{eq:tf}
\bar{t}(G) = 8 \sum_{v \in V(G)} t_G(v) d_G(v), \quad
\bar{f}(G) = 4 \sum_{uv \in E(G)} d_G(u) d_G(v),
\end{equation}
which frequently arise in expressions for $Q$-spectral moments in terms of subgraph counts and vertex degrees.

\begin{lemma} [\cite{S4G,18,T4G}]\label{TG-lem} If $G$ is a simple graph  of order $n$ and size $m$,  then  $$ {S}_{4}(G)=2m+4\varsigma_G(P_3)+8\varsigma_G(C_4), T_{4}(G)={S}_{4}(G)+\bar{t}(G)+\bar{f}(G)+\sum_{i=1}^{n}d^{4}_{i}+4\sum_{i=1}^{n}d^{3}_{i},$$ $$T_{3}(G)=6\varsigma_G(C_3)+\sum_{i=1}^{n}d^{3}_{i}+3\sum_{i=1}^{n}d^{2}_{i}, \hspace{5pt}T_{2}(G)=\sum_{i=1}^{n}d^{2}_{i}+2m\hspace{5pt}\text{and}\hspace{5pt}
	T_{1}(G)=\sum_{i=1}^{n}d_{i}=2m.$$
\end{lemma}

In this section, we denote $$G\cong K_1\vee (C_{k_1}\cup C_{k_2}\cup\cdots \cup C_{k_t}\cup qK_2\cup sK_1), \,\, F^*\cong K_1\vee (P_{l_1}\cup P_{l_2}\cup\cdots \cup P_{l_q}\cup sK_1)$$  $$\text{and} \quad  \widetilde{F}\cong K_1\vee (C_{k'_1}\cup C_{k'_2}\cup\cdots \cup C_{k'_z}\cup P_{l'_1}\cup P_{l'_2}\cup\cdots \cup P_{l'_q}\cup sK_1),$$ where $t,z,q,s\ge 1, l_i,l'_i\ge 2$ and $k_i,k'_{i}\ge 3$. These graphs share the same order $n$, size $m$ and  vertex degree sequence. To ease language, for a constant $c$, we write the multisets $\{k_i~:~ k_i=c, 1\le i \le t \}$ and $\{k'_j~:~ k'_j=c, 1\le j \le z \}$ as $\{k_i~:~ k_i=c\}$ and $\{k'_j~:~ k'_j=c\}$, respectively.

In the following lemmas, the previous setting is assumed by default. In particular, $G$, $F^*$ and $\widetilde{F}$ always denote the graphs specified above.

\begin{lemma}\label{T4G-lem}The following equalities hold.
	\begin{itemize}
		\item[(i)] $\varsigma_{G}(P_3)=3(n-1-s)-4q+\binom{n-1}{2};$
		\item[(ii)] $\varsigma_{G}(C_4)=n-1-2q+\varsigma_{G-v_1}(C_4)=n-2q-s-1+
|\{k_i~:~ k_i=4\}|.$ In particular, if $k_i \ne 4$ for $1\le i\le t$, then $\varsigma_G(C_4)=n-2q-s-1$;
\item[(iii)] $\bar{t}(G)=8\big((n+5)(n-s-1)-q(n+7)+9|\{k_i~:~ k_i=3\}|\big).$ In particular, if $k_i \ne 3$  for $1\le i\le t$, then $\bar{t}(G)=8\big((n+5)(n-s-1)-q(n+7)\big)$;
    \item[(iv)] $\bar{f}(G)=4\big(3(n+2)(n-1-2q)+4qn-s(2n+7)\big).$
	\end{itemize}
\end{lemma}

\begin{proof} First, we have
 $$\varsigma_{G}(P_3)=\sum_{i=1}^{t}k_i+2\sum_{i=1}^{t}k_i+\binom{n-1}{2}+2q
 =3(n-s-1)-4q+\binom{n-1}{2}.$$
  Secondly,
  $$ \varsigma_{G}(C_4)=\sum_{i=1}^{t}k_i+\varsigma_{G-v_1}(C_4)=n-1-2q-s+|\{k_i~:~ k_i=4\}|.  $$
  For convenience, we suppose $k_1\ge k_2\ge \cdots \ge k_b\ge 4 >k_{b+1}=\cdots=k_t=3$, where $0\le b\le t$. Using the first equality of~\eqref{eq:tf}, we compute
   \begin{eqnarray*}\label{tbarG-e}
\sum_{v\in V(G)}t_G(v)d_G(v)&=&4q+(\sum_{i=1}^{t}k_i+q)(n-1)+6\sum_{i=1}^{b}k_i+9\sum_{i=b+1}^{t}k_i    \nonumber\\
&=&4q+(n-1-q-s)(n-1)+6\big(n-1-2q-s-3(t-b)\big)+27(t-b) \nonumber\\
&=&(n+5)(n-s-1)-q(n+7)+9(t-b) \nonumber\\
&=&(n+5)(n-s-1)-q(n+7)+9|\{k_i~:~ k_i=3\}|. \nonumber
\end{eqnarray*}
Thus, (iii) holds.

Finally, we deal with (iv). From the second equality of~\eqref{eq:tf} and
   \begin{eqnarray*}\label{fbarG-e}
\sum_{uv\in E(G)}d_G(u)d_G(v)
&=&9(n-1-2q)+4q+3(n-1)(n-1-2q-s)+4q(n-1)+s(n-1)   \nonumber\\
&=&3(n+2)(n-1-2q)+4qn-s(2n+7), \nonumber
\end{eqnarray*}
we obtain the desired equality.
\end{proof}

In what follows we deal with particular spectral moments or $Q$-spectral moments of $F^*$ and $G$.

\begin{lemma}\label{S4F*-lem} We have ${S}_4(F^*)={S}_4(G)-8|\{k_i~:~ k_i=4\}|.$ In particular, if $k_i \ne 4$  for $1\le i\le t$, then ${S}_4(F^*)={S}_4(G)$.
\end{lemma}
\begin{proof}We suppose $l_1\ge l_2\ge \cdots \ge l_r \ge 3>l_{r+1}=\cdots =l_q=2$ ($1\le r \le q$).

By employing Lemma~\ref{T4G-lem}, we find
   \begin{eqnarray*}
\varsigma_{F^*}(P_3)&=&\binom{n-1}{2}+2\sum_{i=1}^{q}(l_i-1)+\sum_{i=1}^{r}(l_i-2)    \nonumber\\
&=&\binom{n-1}{2}+2(n-q-s-1)+(n-2q-s-1)    \nonumber\\
&=&\binom{n-1}{2}+3(n-s-1)-4q\nonumber\\
&=&\varsigma_{G}(P_3).\nonumber
\end{eqnarray*}
By the same lemma,
$$\varsigma_{F^*}(C_4)=\sum_{i=1}^{r}(l_i-2)=n-2q-s-1=\varsigma_{G}(C_4)-|\{k_i~:~ k_i=4 \}|,$$
which, together with Lemma~\ref{TG-lem} leads to   the desired result.
\end{proof}

\begin{lemma}\label{T4F*-lem} For $F^*\cong K_1\vee (P_{l_1}\cup P_{l_2}\cup\cdots \cup P_{l_q}\cup sK_1)$, with $l_1\ge l_2\ge \cdots \ge l_r \ge 3>l_{r+1}=\cdots =l_q=2$ $(1\le r \le q)$, we have
$$T_4(F^*)=T_4(G)-8|\{k_i~:~ k_i=4\}|-72|\{k_i~:~ k_i=3\}|-4r<T_4(G).$$
\end{lemma}

\begin{proof} By employing Lemma \ref{T4G-lem}, we find
   \begin{eqnarray*}
\bar{t}(F^*)&=&8\sum_{v\in V(F^*)}t_{F^*}(v)d_{F^*}(v)    \nonumber\\
&=&8\big(4q+6(n-1-2q-s)+(n-q-s-1)(n-1)\big)   \nonumber\\
&=&8\big((n+5)(n-s-1)-q(n+7)\big)\nonumber\\
&=&\bar{t}(G)-72|\{k_i~:~ k_i=3\}|\nonumber
\end{eqnarray*}
and
   \begin{eqnarray*}
\bar{f}(F^*)&=&4\sum_{uv\in E(F^*)}d_{F^*}(u)d_{F^*}(v)   \nonumber\\
&=&4\big(4q(n-1)+3(n-1)(n-1-s-2q)+s(n-1)+12r+4(q-r)+9(\sum_{i=1}^{r}l_i-3r)\big)   \nonumber\\
&=&4\big(4q(n-1)+3(n-1)(n-1-s-2q)+s(n-1)+4q-r+9(n-1-s-2q)\big)\nonumber\\
&=&4\big(4qn+3(n+2)(n-1-2q)-s(2n+7)-r\big)\nonumber\\
&=&\bar{f}(G)-4r.\nonumber
\end{eqnarray*}
Now, the desired result follow from Lemmas~\ref{TG-lem} and~\ref{S4F*-lem}.
\end{proof}

Next we deal with particular spectral moments or $Q$-spectral moments of $\widetilde{F}$ and $G$.

\begin{lemma}\label{S4F**-lem} We have ${S}_4(\widetilde{F})-{S}_4(G)=8\big(|\{k'_j~:~ k'_j=4\}|-|\{k_i~:~ k_i=4\}|\big).$
\end{lemma}
\begin{proof}We suppose $l'_1\ge l'_2\ge \cdots \ge l'_r \ge 3>l'_{r+1}=\cdots =l'_q=2$ ($0\le r \le q$).
From Lemma \ref{T4G-lem}, we have
   \begin{eqnarray*}
\varsigma_{\widetilde{F}}(P_3)&=&\sum_{i=1}^{z}k'_{i}+2\sum_{i=1}^{z}k'_{i}+\binom{n-1}{2}
+2\sum_{i=1}^{q}(l'_i-1)+\sum_{i=1}^{r}(l'_i-2)    \nonumber\\
&=&3\sum_{i=1}^{z}k'_{i}+\binom{n-1}{2}
+2\sum_{i=1}^{q}(l'_i-1)+n-1-s-2q-\sum_{i=1}^{z}k'_{i}   \nonumber\\
&=&3(n-s-1)-4q+\binom{n-1}{2}\nonumber\\
&=&\varsigma_{G}(P_3).\nonumber
\end{eqnarray*}
Moreover,
   \begin{eqnarray*}
\varsigma_{\widetilde{F}}(C_4)&=&\sum_{i=1}^{z}k'_i+|\{k'_j~:~ k'_j=4\}| +\sum_{i=1}^{r}(l'_i-2)    \nonumber\\
&=&\sum_{i=1}^{z}k'_i+|\{k'_j~:~ k'_j=4\}|
+n-2q-s-1-\sum_{i=1}^{z}k'_i  \nonumber\\
&=&n-2q-s-1+|\{k'_j~:~ k'_j=4\}|.    \nonumber
\end{eqnarray*}

From Lemma \ref{TG-lem}, we have ${S}_4(\widetilde{F})=2m(\widetilde{F})+4\varsigma_{\widetilde{F}}(P_3)+8\varsigma_{\widetilde{F}}(C_4)$. Together with obtained equalities and Lemma~\ref{T4G-lem}, this leads to the desired result.
\end{proof}

\begin{lemma}\label{T4F**-lem} Let $\widetilde{F}\cong K_1\vee (C_{k'_1}\cup C_{k'_2}\cup \cdots \cup C_{k'_z}\cup P_{l'_1}\cup P_{l'_2}\cup\cdots \cup P_{l'_q}\cup sK_1)$ with $l'_1\ge l'_2\ge \cdots \ge l'_r \ge 3>l'_{r+1}=l'_{r+2}=\cdots =l'_q=2$ $(0\le r \le q)$ and $G$  as stated before. Then
\begin{align*}
T_4(\widetilde{F})-T_4(G)=8\big(|\{k'_j~:~ k'_j=4\}|-|\{k_i~:~ k_i=4\}|  \big)+72\big(|\{k'_j~:~ k'_j=3\}|-|\{k_i~:~ k_i=3\}|\big)-4r. \vspace*{-0.8cm}
\end{align*}
 In particular, if $|\{k'_j~:~ k'_j=3\}|-|\{k_i~:~ k_i=3\}|=0$, then
\begin{align*}
T_4(\widetilde{F})-T_4(G)=8\big(|\{k'_j~:~ k'_j=4\}|-|\{k_i~:~ k_i=4\}| \big)-4r.
\end{align*}
\end{lemma}

\begin{proof}
We suppose $k'_1\ge k'_2\ge \cdots \ge k'_{z_1}\ge 4 >k'_{z_1+1}=\cdots=k'_z=3$, where $0\le z_1\le z$.

By employing Lemma~\ref{T4G-lem}, we compute
   \begin{eqnarray*}
\bar{t}(\widetilde{F})&=&8\sum_{v\in V(\widetilde{F})}t_{\widetilde{F}}(v)d_{\widetilde{F}}(v)    \nonumber\\
&=&8\Big(6\sum_{i=1}^{z_{1}}k'_i+27(z-z_1)+4q+6\sum_{i=1}^{q}(l'_i-2)+(n-q-s-1)(n-1)\Big) \\
&=&8\big(6(n-1-s)+9(z-z_1)-8q+(n-q-s-1)(n-1)\big)   \nonumber\\
&=&8\big((n+5)(n-s-1)-q(n+7)+9|\{k'_j~:~ k'_j=3\}|\big)\nonumber\\
&=&\bar{t}(G)+72(|\{k'_j~:~ k'_j=3\}|-|\{k_i~:~ k_i=3\}|)\nonumber
\end{eqnarray*}
and
   \begin{eqnarray*}
\bar{f}(\widetilde{F})&=&4\sum_{uv\in E(\widetilde{F})}d_{\widetilde{F}}(u)d_{\widetilde{F}}(v)   \nonumber\\
&=&4\Big( 9\sum_{i=1}^{z}k'_i+3(n-1)\sum_{i=1}^{z}k'_i+4q(n-1)+3(n-1)\sum_{i=1}^{q}(l'_i-2) \nonumber\\
&&+s(n-1)+12r+4(q-r)+9\sum_{i=1}^{r}(l'_i-3) \Big)   \nonumber\\
&=&4\big(4qn-r+3(n+2)(n-2q-s-1)+s(n-1)\big)\nonumber\\
&=&4\big(4qn+3(n+2)(n-2q-1)-s(2n+7)-r\big)\nonumber\\
&=&\bar{f}(G)-4r.\nonumber
\end{eqnarray*}
The result follows from Lemmas \ref{TG-lem} and \ref{S4F**-lem}.
\end{proof}

\section{Proof  of Theorem~\ref{11t}}\label{sec5}
 We quote the  known result.
\begin{lemma}[\cite{Ye2025Signless}] \label{PathCycle-lem} Let $H\cong K_1\vee (P_l \cup H_1)$, where $l\ge 4$ and $H_1$ is any multigraph.
\begin{itemize}
\item[(i)] If $l=4$, then
$$\chi_1(H)<\chi_1(K_{1}\vee (C_{2}\cup K_{2}\cup H_1)),$$
where $C_2$ is the digon;
\item[(ii)] If $l\ge 5$, then
$$\chi_1(H)<\chi_1(K_{1}\vee (C_{r}\cup P_{l-r}\cup H_1))$$ holds for $3\leq r\leq l-2.$
\end{itemize}
\end{lemma}

Henceforth, $G$ is  as in the formulation of Theorem~\ref{11t}, with $k\ge 4$, and $F$ is a simple graph  that is $Q$-cospectral with~$G$. Their common order is $n$.

On the basis of Lemma~\ref{SQG-lem} and Remark \ref{r3.2}, we deduce the following setting: $$\chi_1(F)>n>5>\chi_2(F)>1>\chi_n(F)>0.$$
We prove the following lemma comparing the structure of $F$ and $G$.

\begin{lemma}\label{l5.1} If $n\ge 21$ and $k\ge 4$, then $F$ and $G$ share the same vertex degrees and $\varsigma_F(C_3)=\varsigma_G(C_3)$.
\end{lemma}
\begin{proof} From Lemma~\ref{d1-lem}, we know that $F$ is connected with $d_{1}(F)=n-1$ and $d_{2}(F)\le 4$ (this equality will be frequently used) when $n\ge 16$. According to the convention of Section~\ref{sec2},  $v_1$ denotes a vertex with maximum degree in $F$.  	By inserting $n_4(G)=0$, $n_3(G)=k=n-1-2q-s$, $n_2(G)=2q$, $n_1(G)=s$  and  $d_1(G)=d_{1}(F)=n-1$ in Lemma~\ref{TG-lem}, we arrive at
	\begin{equation}\label{e4.1}\left\{
		\begin{array}{ll}  n_{1}(F)=s-n_{4}(F),\\
			               n_{2}(F)=2q+3n_{4}(F),\\
                            n_{3}(F)=n-1-2q-s-3n_{4}(F).
		\end{array}
		\right.\end{equation}
By Lemma \ref{n4-lem}, we find $n_{4}(F)\le 1$. If $n_{4}(F)=0$, then by \eqref{e4.1}, $F$ and $G$ share the same degrees. Moreover, $\varsigma_F(C_3)=\varsigma_G(C_3)$ follows from Lemma \ref{TG-lem}.

Let $n_{4}(F)=1$ and denote by $w$ the unique $4$-vertex in $F$. From \eqref{e4.1}, we have
    \begin{equation*}\label{e4.2}\left\{
		\begin{array}{ll}  n_{1}(F)=s-1,\\
			               n_{2}(F)=2q+3,\\
                            n_{3}(F)=n-2q-s-4=k-3.
		\end{array}
		\right.\end{equation*}

Then by Lemma \ref{TG-lem}, we have
\begin{align}\label{e4.3}
\varsigma_G(C_3)=\varsigma_F(C_3)+1.
\end{align}

By Lemma \ref{n4-lem},  the $4$-vertex $w$ is adjacent to at most one 3-vertex. Together with Lemmas~\ref{DeleteEdge-lem} and \ref{leq5-lem-2}, this structural observation leads to the conclusion that every component of  $F-v_1$ is isomorphic to either $Z_r~(4\le r\le 5)$ or $P_l~(l\ge 1)$. Moreover, we conclude the following:
\begin{itemize}
	\item[(a)] From $n_{4}(F)=1$, we deduce that there is exactly one copy of $Z_r~(4\le r\le 5)$ in $F-v_1$. It gives rise to three $2$-vertices in $F$.
	
	\item[(b)] From $n_{2}(F)=2q+3$ and  (a), we deduce that there are  exactly $q$ disjoint paths with length at least $2$ in $F-v_1$. We denote them by $P_{l_1}, P_{l_2}, \ldots, P_{l_q}~(l_i\ge 2)$.
	
	\item[(c)] Moreover, $n_{1}(F)=s-1$ implies the existence of $s-1$ copies of $P_1$ in $F-v_1$.
\end{itemize}

Note that $\sum_{i=1}^{q}l_i=n-s-r$. From (a), (b) and (c), we deduce  $$\varsigma_F(C_3)= (r-1)+\sum_{i=1}^{q}(l_i-1)=n-s-q-1.$$  On the other hand, since $k\ge 4$, we have $$\varsigma_G(C_3)=n-2q-s-1+q=n-s-q-1.$$ Therefore, $\varsigma_G(C_3)\ne \varsigma_F(C_3)+1$, which contravenes the  equality~\eqref{e4.3}. This completes the proof.
\end{proof}

The next theorem plays a crucial role in this section.

\begin{theorem}\label{t5.1} For $n\ge 21$, the following statements hold true.
 \begin{itemize}
  \item[(i)]If $4\le k\le 5$, then $F\cong G$ (which means that $G$ is DQS);
  \item[(ii)]If $k\ge 6$ and $q=1$, then $F\cong G$ (which means that $G$ is DQS);
  \item[(iii)]If $k\ge 6$ and $q\ge 2$, then $$F\cong G \quad \text{or} \quad F\cong K_1\vee (C_{4}\cup P_{k-3}\cup P_{3}\cup (q-2)K_2 \cup sK_1).$$
 \end{itemize}
\end{theorem}

\begin{proof} By Lemma \ref{l5.1}, $d_{1}(F)=n-1$, so we may suppose $d_F(v_1)=n-1$. The co-degree condition  also leads to the conclusion that every component of  $F-v_1$ is isomorphic to either $C_p~(3\le p \le k)$ or $P_l~(l\ge 1)$. Besides, since $\chi_2(F)<5$, $F$ cannot contain $K_1\vee (C_{p_1}\cup C_{p_2})$ as a subgraph by Lemma \ref{leq5-lem-2}. Hence, at most one component of  $F-v_1$ is a cycle. Moreover, $n_{2}(F)=2q$ and  $n_{1}(F)=s$ imply that there are exactly $q$ disjoint paths with length at least 2  and $s$ copies of $K_1$ in  $F-v_1$. Therefore, the only structural possibilities for $F$ are $$K_1\vee (C_{p}\cup P_{l_1}\cup P_{l_2}\cup\cdots \cup P_{l_q}\cup sK_1) \quad \text{or} \quad K_1\vee (P_{l_1}\cup P_{l_2}\cup\cdots \cup P_{l_q}\cup sK_1)$$ where $3\le p\le k$ and $l_1\ge l_2\ge \cdots \ge l_r \ge 3>l_{r+1}=\cdots =l_q=2$ $(0\le r \le q)$. Noting that $\varsigma_G(C_3)=\varsigma_F(C_3)$ (by Lemma \ref{l5.1}) and $k\ge 4$, we find $p\ne 3$.  Moreover,  by Lemma \ref{T4F*-lem}, we know that the latter graph is not $Q$-cospectral with $G$. Hence, \begin{align}\label{e4.4}
F\cong G \quad \text{or} \quad F\cong K_1\vee (C_{p}\cup P_{l_1}\cup P_{l_2}\cup\cdots \cup P_{l_q}\cup sK_1).
\end{align}

 If $k=4$, we immediately obtain $F\cong G$. If $k=5$ and $F\not\cong G$, then $F\cong K_1\vee (C_4\cup P_3\cup (q-1)K_2\cup sK_1).$ However, this implies $T_4(G)\ne T_4(F)$ by Lemma \ref{T4F**-lem}, which is impossible. Thus $F\cong G$.

Finally, we deal with the case  $k\ge 6$. Combining \eqref{e4.4} and Lemma \ref{T4F**-lem}, we find  $F\cong G$ if $q=1$. For $q\ge 2$, we eventually have
\begin{align*}
F\cong K_1\vee (C_{4}\cup P_{l_1}\cup P_{l_2}\cup (q-2)K_2 \cup sK_1),
\end{align*}
 where $l_1+l_2=k\ge 6$ and $l_1,l_2\ge 3$. Next we claim that $l_1=3$ or $l_2=3$. Otherwise by Lemmas \ref{PathCycle-lem} and \ref{GG*-lem},
 $$\chi_1(K_1\vee (C_{4}\cup P_{l_1}\cup P_{l_2}\cup (q-2)K_2 \cup sK_1))<
 \chi_1(K_1\vee (C_{4}\cup C_{l_1-2}\cup C_{l_2-2}\cup qK_2 \cup sK_1))=\chi_1(G),
 $$
violating the assumption on $Q$-cospectrality. This completes the proof.
\end{proof}

We denote $$\widetilde{F}\cong K_1\vee (C_{4}\cup P_{k-3}\cup P_{3}\cup (q-2)K_2 \cup sK_1), \, \text{where} \,\, k\ge 6 \,\, \text{and}\,\, q\ge 2,$$
and recall that $m_{G}(I^*)$ denotes the number of the eigenvalues in the interval $I^*$ in~$S_Q(G)$ and $m_{G}(\rho)$ denotes the multiplicity of the eigenvalue~$\rho$ in~$S_Q(G)$.

Now, we prove the theorem.

\medskip\noindent{\em \textbf{Proof of  Theorem \ref{11t}.}} According to  above established arguments, we suppose that $k\ge 6$ and $q\geq 2$. We also set $n\geq 21$ (as in the formulation of the theorem). By Theorem~\ref{t5.1}, it is sufficient to  prove that the  graph $F$, introduced at the beginning of this section, is isomorphic to $G$ if $k$ is odd. For a contradiction, we suppose that $F\not \cong G$. Then by Theorem~\ref{t5.1}, we have $F \cong \widetilde{F}$. The vertex with degree $n-1$ in $\widetilde{F}$ is denoted by~$v_1$.

By Lemma~\ref{mul-0-lem}, $m_{\widetilde{F}-v_1}(0)=3+q-2+s=q+s+1$. Then by Lemma \ref{DeleteVertex-lem}, we get $$m_{\widetilde{F}}((0,1])\ge q+s+1.$$
On the other hand, since $k$ is odd, by Lemma \ref{SQG-lem} and Remark \ref{r3.2}, we know $$ m_{G}((0,1])=s-1+q+1=q+s<m_{\widetilde{F}}((0,1]).$$ Therefore, $G$ and $\widetilde{F}$ are not $Q$-cospectral, and this contraction  completes the proof.
\qed

\section{Proof  of Theorem~\ref{12t}}\label{sec6}

Let $G$ be as in the formulation of Theorem~\ref{12t}. As before, $F$ stands for a  graph $Q$-cospectral with~$G$. Also, Lemma~\ref{SQG-lem} and Remark \ref{r3.2} ensure the following: $$\chi_1(F)>n>5=\chi_2(F)>1>\chi_n(F)>0.$$
From Lemma~\ref{d1-d2-lem}, we know that when $n\ge 12$, $F$ is connected with $d_{2}(F)\le 4$ and $d_{1}(F)\ge n-3$.

We proceed by proving the following lemma concerning the largest vertex degrees. A more refined statement is given in the subsequent lemma.

\begin{lemma}\label{l6.1} If $n\ge 33$, then $d_{1}(F)=d_1(G)=n-1$.
\end{lemma}
\begin{proof} By Lemma~\ref{d1-d2-lem}, we have $d_{1}(F)\ge n-3\ge 11$, when $n\ge 14$. By Lemma \ref{23l}, we have $n<\chi_1(F)\le d_{1}(F)+3$, and thus $d_{1}(F)\ge n-2$. Suppose that $d_{1}(F)=n-2$.
	
	Lemma \ref{TG-lem} yields the following system of equations:
	\begin{equation}\label{e61}\left\{
		\begin{array}{ll}\vspace{0.35cm}
n_{1}(F)+n_{4}(F)=n_1(G)+n_4(G)+\frac{d_1(G)(d_1(G)-5)}{2}-\frac{d_{1}(F)(d_{1}(F)-5)}{2},\\\vspace{0.35cm}
n_{3}(F)+3n_{4}(F)=n_3(G)+3n_4(G)+\frac{d_1(G)(d_1(G)-3)}{2}-\frac{d_{1}(F)(d_{1}(F)-3)}{2},\\\vspace{0.2cm}
n_{2}(F)-3n_{4}(F)=n_2(G)-3n_4(G)-d_1(G)(d_1(G)-4)+d_{1}(F)(d_{1}(F)-4).\\
		\end{array}
		\right.\end{equation}

	By inserting $d_{1}(F)=n-2$, $d_1(G)=n-1$, $n_4(G)=0$,  $n_3(G)=n-1-2q-s$, $n_2(G)=2q$ and $n_1(G)=s$ in \eqref{e61}, we get
	\begin{equation}\label{e62}\left\{
		\begin{array}{ll}  n_{1}(F)=n+s-4-n_{4}(F)\\
			               n_{2}(F)=2q-2n+7+3n_{4}(F),\\
                            n_{3}(F)=2n-2q-s-4-3n_{4}(F).
		\end{array}
		\right.\end{equation}
	Hence,
	\begin{align}\label{e63}0\le n_{3}(F)+n_{2}(F)=3-s.
	\end{align}
	Thus, $s\in \{1,2,3\}$. The three scenarios are discussed separately.
	
	\smallskip\noindent{\textit{Case 1: $s=1$.}}  Remark~\ref{r3.2} implies $m_G(1)\le (n+2q-2)/4$, while by~\eqref{e63}, we also have $n_{3}(F)+n_{2}(F)=2$.
	
	\smallskip\noindent{\textit{Subcase 1.1: $n_{3}(F)=2$ and $n_{2}(F)=0$.}} Here, the system \eqref{e62} leads to $n_{4}(F)=(2n-2q-7)/3$ and $n_{1}(F)=(n+2q-2)/3$. Note that $d_{1}(F)=n-2$. If $n\ge 25$,  Lemma~\ref{eigenvalue1-lem} leads to the impossible scenario: $m_F(1)\ge n_{1}(F)-2>(n+2q-2)/4\ge m_G(1)$.
	
	\smallskip\noindent{\textit{Subcase 1.2: $n_{3}(F)=n_{2}(F)=1$.}} By \eqref{e62}, we have  $n_{4}(F)=(2n-2q-6)/3$ and $n_{1}(F)=(n+2q-3)/3$. Moreover, Lemma~\ref{eigenvalue1-lem} leads to $m_F(1)\ge n_{1}(F)-2>(n+2q-2)/4\ge m_G(1)$, whenever $n\ge 29$.
	
	\smallskip\noindent{\textit{Subcase 1.3: $n_{3}(F)=0$ and $n_{2}(F)=2$.}} From  $n_{1}(F)=(n+2q-4)/3$, we arrive at $m_F(1)\ge n_{1}(F)-2>(n+2q-2)/4\ge m_G(1)$, for  $n\ge 33$.
	
	\smallskip\noindent{\textit{Case 2: $s=2$.}} Remark~\ref{r3.2} gives $m_G(1)\le (n+2q+1)/4$, while~\eqref{e63} leads to $n_{3}(F)+n_{2}(F)=1$. The two subcases, $(n_{3}(F), n_{2}(F))=(1,0)$ and $(n_{3}(F), n_{2}(F))=(0,1)$, are resolved as in the previous part of the proof, i.e., on the basis of Lemma~\ref{eigenvalue1-lem} and \eqref{e62} where we arrive at a contradiction by considering $n \geq 22$ in the former case and $n \geq 18$ in the latter.

	\smallskip\noindent{\textit{Case 3: $s=3$.}}  $m_G(1)\le (n+2q+4)/4$ is deduced from Remark \ref{r3.2}, and~\eqref{e63} implies  $n_{2}(F)=n_{3}(F)=0$.  Lemma~\ref{eigenvalue1-lem} eliminates the possibility $n\ge 19$.

In summary, $d_1(G)=d_{1}(F)$ holds when $n\ge 33$.
\end{proof}

\begin{lemma}\label{l6.2} If $n\ge 33$, then $F$ and $G$ share the same vertex degrees, along with $\varsigma_G(C_3)=\varsigma_F(C_3)$.
\end{lemma}
\begin{proof} Lemma~\ref{l6.1} ensures $d_{1}(F)=n-1$ whenever $n\ge 33$. A corresponding vertex is denoted by $v_1$. Since $d_1(G)=d_{1}(F)$, $n_4(G)=0$, $n_3(G)=n-1-2q-s$, $n_2(G)=2q$ and  $n_1(G)=s$,  from~\eqref{e61}, we deduce
\begin{equation}\label{e64}\left\{
\begin{array}{ll}
                            n_{1}(F)=s-n_{4}(F)\\
			               n_{2}(F)=2q+3n_{4}(F),\\
                            n_{3}(F)=n-2q-s-1-3n_{4}(F).
 \end{array}
 \right.\end{equation}
To achieve our goal, we need  to verify that $n_{4}(F)=0$. Suppose that $n_{4}(F)=b\ge 1$, and let $w_1,w_2,\ldots,w_b$ be the $4$-vertices in $F$. From~\eqref{e64} and Lemma \ref{TG-lem} ($T_3(G)=T_3(F)$), we arrive at
\begin{equation}\label{e65}\left\{
\begin{array}{ll}
                            n_{1}(F)=s-b\\
			               n_{2}(F)=2q+3b,\\
                            n_{3}(F)=n-2q-s-1-3b,\\
                            \varsigma_G(C_3)=\varsigma_F(C_3)+b.
 \end{array}
 \right.\end{equation}
By Lemma \ref{leq5-lem},  every $w_i$ is adjacent to at most one 3-vertex, there is no edge between $w_i$ and $w_j$ ($i\neq j$), and  $N_F(w_i)\cap N_F(w_j)=\{v_1\}$. Combining with Lemmas~\ref{DeleteEdge-lem} and \ref{leq5-lem-2}, these structural observations lead to the conclusion that every component of  $F-v_1$ is isomorphic to either $Z_p~(4\le p\le 5)$, or $C_k~(k\ge 3)$, or $P_l~(l\ge 1)$. Moreover, we deduce the following:

\begin{itemize}
	\item[(a)] There are exactly $b$ copies of $Z_p$ in $F-v_1$, as $n_{4}(F)=b$. We suppose that there are $b_1$ copies of $Z_4$ and $b-b_1$ copies of $Z_5$, which give rise to $4b-b_1$ triangles in $F$.
	
	\item[(b)] Note that $n_{2}(F)=2q+3b$, and from (a) we know that there are  $b$ copies of $Z_p~(4\le p\le 5)$ in $F-v_1$, which gives $3b$ $2$-vertices in $F$. Thus, there are exactly $q$ paths of length at least two in $F-v_1$, say $P_{l_1}, P_{l_2},\ldots,P_{l_q}$ $(l_i\ge 2)$.  They give rise to $\sum_{i=1}^{q}(l_i-1)$ triangles in $F$.

\item[(c)]Further, $n_{1}(F)=s-b$ implies the existence of $s-b$ copies of $P_1$ in $F-v_1$. Of course, they do not contribute any triangle in $F$.
	
	\item[(d)] Finally, the remaining components of $F-v_1$ are cycles which produce at least $$n-2q-s-1-3b-(b-b_1)-\sum_{i=1}^{q}(l_i-2)=n-s-1-4b+b_1-\sum_{i=1}^{q}l_i$$ triangles in $F$.
\end{itemize}

From (a), (b), (c) and (d), we immediately obtain $$\varsigma_F(C_3)\ge (4b-b_1)+ \sum_{i=1}^{q}(l_i-1)+(n-s-1-4b+b_1-\sum_{i=1}^{q}l_i)=n-s-q-1.$$  On the other hand, we know that $$\varsigma_G(C_3)=\sum_{i=1}^{t}k_i+q=n-s-q-1.$$ Therefore, $\varsigma_G(C_3)\not=\varsigma_F(C_3)+b$, which contradicts \eqref{e65}. So, $G$ and $F$ share the same degrees. By Lemma~\ref{TG-lem} ($T_3(G)=T_3(F)$), we deduce $\varsigma_G(C_3)=\varsigma_F(C_3)$.
\end{proof}

We now eliminate all but one structural possibility for $F$.

\begin{lemma}\label{t6.3} For $n \geq 33$, the only structural possibility for $F$ is:
	$$
	K_1 \vee (C_{k'_1} \cup C_{k'_2} \cup \cdots \cup C_{k'_z} \cup P_{l_1} \cup P_{l_2} \cup \cdots \cup P_{l_q}\cup sK_1).
	$$
	where $z\ge 1$, $k'_i \geq 3$, $l_1 \geq l_2 \geq \cdots \geq l_r \geq 3 > l_{r+1} = \cdots = l_q = 2$ $(0 \leq r \leq q)$.
\end{lemma}
\begin{proof} Lemma~\ref{l6.1} ensures $d_{1}(F)=n-1$ when $n\ge 33$; with~$v_1$ in the role of a corresponding vertex. Besides, the co-degree condition (Lemma~\ref{l6.2}) leads to the conclusion that every component of  $F-v_1$ is isomorphic to either $C_k~(k\ge 3)$ or $P_l~(l\ge 1)$.  Moreover, $n_{2}(F)=2q$ implies that there are exactly $q$ disjoint paths of length at least 2 in  $F-v_1$. Therefore,
	$$
	F\cong K_1 \vee (C_{k'_1} \cup C_{k'_2} \cup \cdots \cup C_{k'_z} \cup P_{l_1} \cup P_{l_2} \cup \cdots \cup P_{l_q}\cup sK_1) \quad \text{or} \quad F\cong K_1 \vee (P_{l_1} \cup P_{l_2} \cup \cdots \cup P_{l_q}\cup sK_1),
	$$
where $z\ge 1, k'_i\ge 3, l_i\ge 2$. However, the latter possibility is eliminated by Lemma~\ref{T4F*-lem}.
\end{proof}

As in Section \ref{sec4}, we will write the multiset $\{k'_j~:~ k'_j=c, 1\le j \le z\}$ as $\{k'_j~:~ k'_j=c\}$.

\medskip\noindent{\em \textbf{Proof of Theorem \ref{12t}.}}
We assume that $n\geq 33$ and each $k_i$ $(k_i\ge 4)$ is odd (as in the statement of this theorem). Let $F$ be as in the beginning of this section, meaning that either $F\cong G$ or $F$ is as in the formulation of Lemma~\ref{t6.3}. In what follows we eliminate the latter scenario by the way of contradiction.

\smallskip\noindent{\textit{Case 1: $r=0$.}}
 Then $$F \cong K_1\vee (C_{k'_1} \cup C_{k'_2} \cup \cdots \cup C_{k'_z}\cup qK_2\cup sK_1).$$ Note that $z=t$ because $m_F(5)=z-1=t-1=m_G(5)$ (see Lemma~\ref{SQG-lem}). Moreover, we may suppose $k_{1}\ge k_{2} \ge\cdots \ge k_t\ge 5$, $k'_{1}\ge k'_{2} \ge\cdots \ge k'_t\ge 3$. Observing that $3+2\cos\frac{2\pi}{k_j}$ is increasing with $k_j$, we deduce that $k_1=k'_1$. By excluding the common $Q$-eigenvalues $3+2\cos\frac{2i\pi}{k_1},$ $1\leq i\leq k_1-1$, from the $Q$-spectrum  stated in Lemma~\ref{SQG-lem}, we arrive at $k_2=k'_2$. By repeating this procedure, we deduce $k_j=k'_j$, for all $j$, which yields $F\cong G$.

\smallskip\noindent{\textit{Case 2: $r\ge 1$.}}
 By Lemma \ref{l6.2}, we find $k'_j\ne 3$, for $1\le j\le z$ because $k_i\ge 5$ and $\varsigma_G(C_3)=\varsigma_F(C_3)$. Then, by Lemma \ref{T4F**-lem},  $$T_4(F)-T_4(G)=8|\{k'_j~:~ k'_j=4\}|-4r.$$
 From $T_4(F)-T_4(G)=0$ and $r\ne 0$, we deduce $$|\{k'_j~:~ k'_j=4\}|\ge 1 \quad \text{and} \quad r=2|\{k'_j~:~ k'_j=4\}|\ge 2.$$
 Now, by Lemma~\ref{mul-0-lem}, we have $m_{F-v_1}(0)\ge q+s+1$, and so by Lemma~\ref{DeleteVertex-lem},  $$m_{F}((0,1])\ge q+s+1.$$
On the other hand, since each $k_i$ is odd, combining with Lemma \ref{SQG-lem} and Remark \ref{r3.2}, we obtain $$m_{G}((0,1])=q+s<m_{F}((0,1]).$$ Therefore, $G$ and $F$ are not $Q$-cospectral, contrary to our assumption.
\qed

\section*{Acknowledgements}
This  work was supported by the National Natural Science Foundation of China (No.~12361070) and  the Ministry
	of Science, Technological Development, and Innovation of the Republic of Serbia (No.~451-03-136/2025-03/200104).

\end{document}